\documentclass[12pt]{article}
\usepackage[colorlinks]{hyperref}
\usepackage{color}
\usepackage{graphicx}
\usepackage{graphics}
\usepackage{makeidx}
\usepackage{showidx}
\usepackage{latexsym}
\usepackage{amssymb}
\usepackage{verbatim}
\usepackage{amsmath}
\usepackage{amsthm}
\usepackage{amsfonts}
\usepackage{amssymb,amsmath}
\usepackage{latexsym,amsthm,amscd}
\usepackage[all]{xy}

\newtheorem{prop}{Proposition}[section]
\newtheorem{defn}{Definition}[section]
\newcounter{alphthm}
\newtheorem{proposition}[alphthm]{Proposition}

\newtheorem{rem}{Remark}[section]

\newtheorem{ex}{Example}[section]
\newtheorem{thm}{Theorem}
\newtheorem{lem}{Lemma}[section]
\newcommand{\be}{\begin{equation}}
\newcommand{\ee}{\end{equation}}
\newcommand{\ben}{\begin{enumerate}}
\newcommand{\een}{\end{enumerate}}
\newcommand{\beq}{\begin{eqnarray}}
\newcommand{\eeq}{\end{eqnarray}}
\newcommand{\beqn}{\begin{eqnarray*}}
\newcommand{\eeqn}{\end{eqnarray*}}

\newcommand{\pa}{\partial}
\newcommand{\bpf}{\begin{proof}}
\newcommand{\epf}{\end{proof}}
\newcommand{\bl}{\begin{lem}}
\newcommand{\el}{\end{lem}}
\newcommand{\bp}{\begin{prop}}
\newcommand{\ep}{\end{prop}}
\newcommand{\bd}{\begin{defn}}
\newcommand{\ed}{\end{defn}}
\newcommand{\bt}{\begin{thm}}
\newcommand{\et}{\end{thm}}
\newcommand{\bg}{\begin{ex}}
\newcommand{\eg}{\end{ex}}
\newcommand{\br}{\begin{rem}}
\newcommand{\er}{\end{rem}}
\newcommand{\R}{I\!\! R}

\newcommand\bpr{\begin{prop}}
\newcommand\epr{\end{prop}}

\newcommand\y{\textbf{y}}

\numberwithin{equation}{section}

\title{On the geometry of Zermelo's optimal control trajectories}
\author {Z. Fathi and B. Bidabad \footnote{The corresponding author; bidabad@aut.ac.ir; behroz.bidabad@math.univ-toulouse.fr}}
\date{}
\begin{document}
\maketitle
\begin{abstract}
In the present work, we study the optimal control paths in the Zermelo navigation problem from the geometric and differential equations point of view rather than the optimal control point of view, where the latter has been carried out in our recent work.  Here, we obtain the precise form of the system of ODE where the solutions are optimal trajectories of Zermelo's navigation problem.
Having a precise equation allows optimizing a cost function more accurately and efficiently.
The advantage of these equations is to approximate optimal trajectories in the general case by the first order approximation of external fields $w$. The latter could be solved numerically since we have retrieved simpler equations for these paths. \noindent\\
\end{abstract}

\section{Introduction}

Optimal control problems naturally arise in  engineering, especially in robotics and aerospace, where the path optimization is required. The goal in an optimal control problem is to determine a control function that minimizes the objective function (also called the total cost function)  to minimize physical quantities such as time, route, and/or other relevant costs of traveling.

\par Zermelo's navigation problem is considered to be one of the most well-known time-optimal control problems. Zermelo's navigation is a classic problem in the calculus of variations which dates back to Zermelo's notes in 1931.

\par Consider a moving object (boat, vessel, plane, etc.) that is traveling in a stream of an external field (such as  wind or water current) is to reach a set destination in the shortest time possible. Zermelo's navigation problem refers to the characterization (finding/calculation) of the optimal paths.

\par In Zermelo's navigation problem in the Riemannian setting, the underlying space in which the navigation occurs can be modeled by a Riemannian manifold and an external natural (perturbing) force given by a vector field $w$ on this manifold. The problem of Zermelo navigation in Riemannian geometry is studied in detail in \cite{BRSH}. The problem was further studied with several extensions in the papers \cite{BaTS}, \cite{AEWB} among other essential works. In the said works, a beautiful link to geometry is verified. To wit, it is shown that Zermelo's navigation with weak external fields (vector fields $w$ with $|w|<1$) yield paths which are geodesics of specific Randers type Finsler metrics. We will see more of this later.

\par In~\cite{FB}, the authors studied Zermelo's navigation problem with and without
(moving) obstacles from a metric geometric and optimal control point of view  and with a look towards improving existing computational methods. In particular, we proposed a modification
of the optimization scheme previously considered  in~\cite{LCKTC} by adding a piece-wise constant rotation. This fact was shown to produce exact paths in some cases hence was an improvement upon the existing methods. See~\cite{FB} for further details. In~\cite{BR}, the second named author and Rafie-Rad studied the time-optimal trajectories in a non-obstacle pursuit problem, which geometrically shares similarities with the Zermelo navigation problem.

The authors in \cite{LTZD} solve navigation problems as a minmax
problem without obstacles to obtain the optimal time by using a classical penalty function method.
They then proceed to allow for longer times of travel (which are at most a small percentage over the optimal time) and in presence of fixed obstacles.

The paper \cite{LKTLD} considers an optimal PID control problem subject
to continuous inequality constraints and terminal state equality constraint; it was
shown that the problem can be solved via solving a sequence of nonlinear optimization
problems. An efficient computational method was proposed. It was then
applied to a ship steering control problem. The results obtained show that the
method proposed is reliable and effective.

In \cite{JLYTD}, the authors minimize time of travel with continuous
inequality constraints. The stopping time is determined by a smooth
hypersurface.
Several approaches for finding optimal trajectories using wind forecast data have been introduced in the
literature, including analytical optimal control, \cite{JMRBAE}, \cite{MABE}.

Our main goal in this manuscript is to investigate the systems of ordinary differential equations which produce optimal paths. This is useful since having a precise equation to solve numerically might sometimes be easier than optimizing a cost function.

\par Classically, Zermelo's navigation problem occurs in 2D because it was first introduced to navigate the movement of ships in the sea. However, the formulation is not bounded to 2D. For example, in~\cite{FBN}, the navigation of Dubin's airplane in the presence of fixed and moving obstacles is studied, which is a 3D version of Zermelo's navigation problem in the presence of obstacles.

\par As was alluded to, Zermelo's navigation in Riemannian manifolds is closely tied to Randers type Finsler metrics and their geodesics. We will explain this further in the next section, but for now, it is worth mentioning that these geodesics are completely classified in~\cite{CR} in the case where the resulting Randers type metric is of constant flag curvature.

\bt \label{Th; geod} Let $w$ be a weak dynamic linear vector field perturbing a movement in $ \mathbb{R}^2$.
The differential equation of optimal control paths (geodesics)
 starting from initial position to its destination are given by the equation  \ref{geo1}.
\et
These ODEs are useful since the optimal trajectories in the general case can be approximated by the first order  approximation of external fields $w$. 
\section{Background Material}
\subsection{Finsler structures}\label{sec:Finsler}

\par The basic notions and definitions of Finsler geometry can be found in standard texts and  many papers, so we will not review them in detail. We will stick to the notations and definitions used in~\cite{FB}.

\par To be concise, a Finsler structure on a smooth manifold is a base point and direction dependent norm $F$ on the tangent space which is $C^\infty$ smooth away from the zero section of the tangent bundle and which satisfies the positive homogeneity and convexity properties; namely, if we denote a base point by $x$ and tangent vectors by $\left(x, {\bf y}  \right) \in T_x M$, the Finsler metric $F$ must satisfy
\emph{Regularity away from zero section}: $F$ is $C^\infty$ on the entire slit tangent bundle $TM_0$ which is the tangent bundle minus the zero section.
	\emph{Positive homogeneity}: $F(x,\lambda {\bf y})=\lambda F(x, {\bf y})$ for all $ \lambda >0$.
	 \emph{Strong convexity away from zero section}: the $n\times n$ \emph{Hessian} matrix
	$(g_{ij}):=\dfrac{\partial^2 F^2}{2\partial y^i\partial y^j}$, is positive-definite at every point of $TM_0$.
Given such a manifold $ M $ and Finsler structure $F$ on its tangent space, $TM$. The pair $(M,F)$ is known as a \emph{Finsler space} or a \emph{Finsler manifold}. We emphasize that $F$ depends on the vector ${\bf y} \in \R^n$ as well as on the base point $x \in M$. See the standard textbook~\cite{BCSH} for more details.
\par Among Finsler metrics, Randers type metrics are important natural objects in that they are obtained from Riemannian metrics by adding a linear term namely; a \emph{Randers} type metric is a Finsler metric of the form
$
	F(x, {\bf y}) := \alpha(x, {\bf y}) +  \beta(x, y),
$
where $\alpha(x, y):=\sqrt{a_{ij}(x)y^iy^j}$ is a Riemannian metric and $\beta(x, y):=b_i(x)y^i$ is a one form. The assumption $a(b, b) < 1$ is required to ensure the positivity of $F$. This simple condition also guarantees that the metric is strongly convex. That is, the Hessian $g_{ij}(x, y):= ( \frac{1}{2}F^2)_{y^iy^j}$ is positive definite for all nonzero ${\bf y} = y^i\frac{\pa}{\pa x^i}\in T_xM$.
\par A geodesic is a generalization of the notion of a ``straight line" (locally shortest paths or paths with zero tangential acceleration) from Euclidean spaces (and more generally Riemannian manifolds) to Finsler ones. Just like the Riemannian case and from the calculus of variation point of view, minimizing the energy functional will provide us with the geodesic spray; the geodesic spray is the vector field on $TTM$ whose integral curves are the natural lifts of geodesics to $TM$.

\par To express the differential equations that geodesics solve, one needs to pick a connection. In a Finsler structure, the Finsler norm is not necessarily reversible, that is, opposite vectors might have different lengths. Because of this fact, it is necessary to make the difference between (locally) forward length minimizing and (locally) backward length minimizing curves. Throughout these notes, a Finsler geodesic means a forward geodesic (locally forward length minimizing) and a minimizing geodesic is a forward geodesic that minimizes the forward length between its endpoints.
\par The generalization, to the Finsler setting, of the unit sphere bundle is the indicatrix which is the set $S_F:=\lbrace {\bf y}\in V ;F(\y)=1\rbrace$ (point dependent unit ball in norm). The fibers of the indicatrix are easily verified to be  closed and convex subsets enclosing the origin which never passes through the origin. For simplicity, these will be refered to as the unit tangent spheres where there is no risk of confusion.
\subsection{Geodesics of Randers type metrics}
As was mentioned, the study of  time-optimal paths in Zermelo's navigation problem entails the study of geodesics of a suitably defined Randers type metric. So first, let us review how to compute geodesics of a Randers type metric.
\subsubsection{Geodesic spray and equations}
The geodesics in Riemannian and Finsler settings can be approached in various ways. One way is to consider geodesics (along with the tangent vector) as integral curves of a vector field $G$ on the tangent bundle (in the Riemannian setting) or the slit tangent bundle (in the Finsler setting). This vector field is known as the geodesic spray. It is easy to see that, the geodesic spray $G$ is of the form $ G = y^i \frac{\partial}{\partial x_i} - 2G^i\frac{\partial}{\partial y^i}, $
in which the coefficients $G^i$ are called geodesic spray coefficients.
\par It is, by now, standard that \emph{unit speed} geodesics in a Finsler structure $\left( M, F \right)$ are given by the Finsler geodesic spray with coefficients
\begin{align*}
	G^i = \frac{1}{4}g^{il}\left( \left[ F^2 \right]_{x^ky^l} y^k - \left[ F^2 \right]_{x^l} \right) = \frac{1}{4}g^{il}\left( 2\left( g_{jl} \right)_{x^k} - \left( g_{jk} \right)_{x^l} \right)y^jy^k.
\end{align*}
For more details, see~\cite{BCSH}.

\par Once the geodesic spray coefficients are known (in local coordinates), the geodesic equations (for the unit speed geodesic $\gamma$) can be written in coordinates by
\[
\dot{\gamma} = \left(x,y^i\right), \quad \ddot{\gamma} = \left(\dot{\gamma}, y^i, -2G^i   \right)
\]

\par A coordinate-free equation can be obtained by using a suitable connection. We will touch upon this  below.

\subsection{The Randers metric resulting from Zermelo's navigation}\label{subsec:randers}

It is, by now, standard that the optimal paths in Zermelo's navigation problem are geodesics of a Randers type Finsler metric. This geometric aspect helps us to study the optimal paths from various angles, i.e., from the point of view of Finsler geometry, metric geometry and differential equations, in addition to the point of view of optimal control theory.

Indeed the following provides the geometric picture.
\begin{proposition}[Shen~\cite{Sh-1}]\label{prop:Randers}
	Let $M$ be a smooth manifold, suppose the fast vessel travels from the starting point of the vector $\y$ to its end with unit speed and the external factor $w$ produces an effect on $M$ such that $\Vert w\Vert < 1$. The new Finsler structure $F$ on the tangent bundle that measures time motion is
	\begin{equation}\label{eq:Randers}
		F(\y)=\left( 1-\Vert w\Vert ^2   \right)^{-1}\Big(\big(\langle \y,w\rangle ^2+(1-\Vert w\Vert ^2)\Vert \y\Vert ^2\big)^{\frac{1}{2}} - \langle \y,w\rangle \Big)
	\end{equation}
	where $\Vert \cdot \Vert= \langle \cdot , \cdot \rangle^{\frac{1}{2}}$ is the Riemannian norm.
\end{proposition}
\begin{proof}
	See~\cite{Sh-1} for a proof of Proposition~\ref{prop:Randers}; c.f.~\cite{BRSH}.
\end{proof}

Further details and investigations can be found, among other places, in~\cite{Sh-1} and~\cite{FB}.

\section{Precise ODE that optimal trajectories satisfy}

\par As we saw, the optimal trajectories are indeed forward minimizing geodesics of an specific Randers type Finsler metric. In this section, we aim to compute the geodesic equations for general external forces in the 2D case when the underlying manifold is the flat $\R^2$. Let $u$ be an internal field with constant magnitude $\| u \| = 1$ which derives the ship.

\par In this case, our only free parameter is the angle of the vector field $u$. Suppose $w$ is the external flow force of the form $w = \left( w^1, w^2 \right)$. We assume the external force $w$ is weak, that is, we assume $\| w \| < 1$.

\par With this setup, the movement of the ship is governed by the ODE system
\begin{align*}
	\begin{cases}
		\dot{x} =  \cos(\theta(t)) + w^1.\\
		\dot{y} =  \sin(\theta(t)) + w^2.
	\end{cases}
\end{align*}
The Zermelo navigation problem is to find the control $\theta$ that takes the ship from a given initial point to a given destination in the shortest time.

\par As we mentioned in Section \ref{subsec:randers} (see also \cite{CR,BRSH}), the shortest path is a minimizing geodesic of the Randers type Finsler metric in (\ref{eq:Randers}).

\par From (\ref{eq:Randers}), it is straightforward to see that the defining Riemannian metric $ a $ and 1-form $ b $ of the Randers metric are given by
\begin{align*}
a_{ij}=\frac{ \rho\delta_{ij}+w_iw_j}{\rho^2},\quad b_i=-\frac{w_i}{\rho},
\end{align*}
where $w_i:=\delta_{ij}w^j $ and where
\begin{align*}
\rho = 1 - \Vert w \Vert^2 &= 1 - (w^1)^2 - (w^2)^2.
\end{align*}

\par So in our case, we have
\begin{align*}
b_1 = - \frac{w_1}{\rho} = - \frac{w_1}{1 - (w^1)^2 - (w^2)^2},\quad
b_2 = - \frac{w_2}{\rho} = - \frac{w_2}{1 - (w^1)^2 - (w^2)^2}
\end{align*}
and
\begin{align*}
(a)_{2\times 2} &=\frac{1}{\rho^2} \begin{pmatrix} \rho + (w_1)^2 & w_1w_2 \\ w_1w_2 & \rho + (w_2)^2 \end{pmatrix} =\frac{1}{\rho^2} \begin{pmatrix} 1 - (w_2)^2 & w_1w_2 \\ w_1w_2 & 1 - (w_1)^2 \end{pmatrix}.
\end{align*}
Now we compute the geodesic spray coefficients. Notice that for us the background Riemannian metric is $\delta_{ij}$ with vanishing Christoffel symbols. Therefore the Riemannian geodesic spray coefficients are
\begin{align*}
\mathcal{G}^i = 0.
\end{align*}
A curve $ \Gamma:[0,t]\rightarrow \mathbb{R}^2 $ will be a geodesic of the Randers metric  $ F $ if it satisfies the geodesic equation
\begin{align*}
\ddot{\Gamma}^i+2G^i(\Gamma,\dot{\Gamma})=\frac{d}{dt}(\ln F(\dot{\Gamma}))\dot{\Gamma}^i,
\end{align*}
(see \cite{BCSH}). The geodesic coefficients of $ F $ are related to those of the Riemannian metric $ a $ by $ (11.3.12) $ of \cite{BCSH}. As a result the geodesic coefficients of $ F $ are related to the those of $ g_{_{Euc}} $ by
\begin{align*}
G^i=\mathcal{G}^i+\zeta^i,
\end{align*}
where
\begin{align*}
\zeta^i &= \frac{1}{4}\left( \frac{y^i}{F} - w^i \right) \left( 2FS_0 - L_{00} - F^2L_{ww}\right) - \frac{1}{4}F^2\left( S^i + T^i\right) - \frac{1}{2}FC^i_0 .
\end{align*}
Now we wish to compute these constituent parts in our problem presisely. Following \cite{CR} and notations therein, we set
\begin{align*}
	L_{ij}=w_{i:j}+w_{j:i},\:\: S_{i}=w^sL_{si},\:\: C_{ij}=w_{i:j}-w_{j:i},\:\: T_i=w^sC_{si},
\end{align*}
In which, the colon '$ : $' denotes covariant differentiation so that $w_{i:j}=w_{i,x^j}-w_s \gamma ^s_{ij} $ where $ \gamma ^s_{ij} $ denotes the Christoffel symbols of $g_{_{Euc}}$. Indices on these tensors are raised with the inverse of $ g_{_{Euc}} $. For example, $S^i=\delta^{ij}S_j$.
As before the subscript $ 0 $ denotes contraction with $\bf y$, $ C^i_0=\delta^{ij}C_{jk}y^k $.
Finally, $L_{_{ww}}=w^iw^jL_{ij} $.
\begin{align*}
L_{11} = 2 {w_1}_{x_1}, \quad L_{22} = 2 {w_2}_{\substack{\\x_2}},  \quad L_{12} = L_{21} = {w_1}_{\substack{\\x_2}} + {w_2}_{\substack{\\x_1}}
\end{align*}
\begin{align*}
C_{11} = C_{22} = 0 ,\quad C_{12} = {w_1}_{\substack{\\x_2}} - {w_2}_{\substack{\\x_1}}, \quad C_{21} = {w_2}_{\substack{\\x_1}} - {w_1}_{\substack{\\x_2}}
\end{align*}
\begin{align*}
S_1 = w^1L_{11} + w^2 L_{21}= 2w^1{w^1}_{\substack{\\x_1}} + w^2({w^1}_{\substack{\\x_2}} + {w^2}_{\substack{\\x_1}}).
\end{align*}
\begin{align*}
S_2 = w^1L_{21} + w^2 L_{22} = w^1({w^1}_{\substack{\\x_2}} + {w^1}_{\substack{\\x_1}} ) + 2w^2{w^2}_{\substack{\\x_2}} .
\end{align*}
\begin{align*}
T_1 = w^1C_{11} + w^2C_{21} = w^2({w^2}_{\substack{\\x_1}}  - {w^1}_{\substack{\\x_2}} ).
\end{align*}
\begin{align*}
\quad T_2 = w^1C_{12} + w^2C_{22} = w^1({w^1}_{\substack{\\x_2}}  - {w^2}_{\substack{\\x_1}} ).
\end{align*}
Now for the tangent vector $\mathbf{y} = \left(y^1, y^2 \right)$, we have
\begin{align*}
S_0 &= S_1 y^1 + S_2 y^2 \\ &= \left[ 2w^1{w^1}_{\substack{\\x_1}}  + w^2({w^1}_{\substack{\\x_2}}  + {w^2}_{\substack{\\x_1}} ) \right] y^1 \\ &+ \left[ w^1({w^1}_{\substack{\\x_2}}  + {w^2}_{\substack{\\x_1}} ) + 2w^2{w^2}_{\substack{\\x_2}}  \right] y^2,
\end{align*}
for $ C_0^i $ we have
\begin{align*}
&C_0^1=\delta^{11}C_{12}y^2=({w^1}_{\substack{\\x_2}} - {w^2}_{\substack{\\x_1}} ) y^2\\
&C_0^2=\delta^{22}C_{21}y^1=({w^2}_{\substack{\\x_1}} - {w^1}_{\substack{\\x_2}} ) y^1,
\end{align*}
and the longest ones
\begin{align*}
L_{00} &= L_{11} y^1y^1 + L_{22}y^2y^2 + 2L_{12}y^1y^2,\\
 &= 2{w^1}_{\substack{\\x_1}} (y^1)^2 + 2({w^1}_{\substack{\\x_2}}  + {w^2}_{\substack{\\x_1}} ) y^1 y^2 + 2 {w^2}_{\substack{\\x_2}}  (y^2)^2.
\end{align*}
\begin{align*}
L_{ww} &= w_1^2L_{11} + 2w_1w_2L_{12} + w_2^2L_{22}, \notag \\ &=2 (w^1)^2 {w^1}_{\substack{\\x_1}}  + 2w^1w^2({w^1}_{\substack{\\x_2}}  + {w^2}_{\substack{\\x_1}} ) + 2(w^2)^2 {w^2}_{\substack{\\x_2}}  \notag.
\end{align*}
Now we can write
\begin{align*}
G^i = \mathcal{G}^i + \zeta^i = \zeta^i,
\end{align*}
and
\begin{align*}
\zeta^i &= \frac{1}{4}\left( \frac{y^i}{F} - w^i \right) \left( 2FS_0 - L_{00} - F^2L_{ww}\right) - \frac{1}{4}F^2\left( S^i + T^i\right) - \frac{1}{2}FC^i_0 .
\end{align*}
Notice that, since the background metric is Euclidean, there is no difference between lower and upper indices. Now we need to do the calculations and simplify.
Let us compute the hardest term we need.
\begin{align*}
2FS_0 - L_{00}- F^2L_{ww} = &2F [ 2w^1{w^1}_{x_1}  + w^2({w^1}_{x_2}  + {w^2}_{x_1} )] y^1\\
&+ 2F[ w^1({w^1}_{x_2}  + {w^2}_{x_1} ) + 2w^2 {w^2}_{x_2}  ] y^2\\
&- 2{w^1}_{x_1} (y^1)^2 - 2({w^1}_{x_2}  + {w^2}_{x_1} ) y^1 y^2 - 2 {w^2}_{x_2}  (y^2)^2\\
&- F^2 ( 2 (w^1)^2 {w^1}_{x_1} + 2w^1w^2 ({w^1}_{x_2}  + {w^2}_{x_1} )\\
&+ 2(w^2)^2 {w^2}_{x_2}).
\end{align*}
and
\begin{align}\label{eq:s-and-t-1}
S_1 + T_1 &= 2w^1 {w^1}_{\substack{\\x_1}}  + 2w^2 {w^2}_{\substack{\\x_1}}  \\ \notag &= ((w^1)^2 + (w^2)^2)_{x_1}.
\end{align}
\begin{align}\label{eq:s-and-t-2}
S_2 + T_2 &= 2w^1 {w^1}_{\substack{\\x_2}}  + 2w^2{w^2}_{\substack{\\x_2}}  \\ \notag &= ((w^1)^2 + (w^2)^2)_{x_2}.
\end{align}
We have all the terms computed in this case. As one can observe, the terms involved become very tedious hence for a general external vector field $w$, the geodesic equations will become very complicated and consequently impede one's ability to  solve for optimal control trajectories numerically.

\par One remedy would be to linearize the geodesic equations (concerning $\Gamma$) and then try to solve them, but this is also difficult since as we saw  coefficients are very tedious.  A way out is to approximate the optimal trajectories by approximating the external field $w$. This is what we wish to carry out in the next section.

\section{Geodesic equations for affine nature fields}
In order to obtain more sensible equations which would be easier to solve numerically, we can first replace the external field $w$ (we also call it the nature field sometimes) with its first order approximation.

\par Indeed, on small enough rectangles in $\R^2$, we can consider the first order approximation of the weak external field $w$. Since $w$ is a weak external field (i.e. $\| w \| < 1$), its first order approximation is also weak provided that we make the domain smaller if necessary.
In this section, we will hence consider an affine weak external field and wish to obtain the optimal trajectories in Zermelo's navigation problem under this affine approximation. The original optimal trajectories will later be shown to be obtained in a limiting process.

\par Assume $w^1$ and $w^2$ are affine functions on a rectangular region $A \le x_1 \le B$ and $C \le x_2 \le D$. Therefore,
\begin{align*}
w^1 =c_1 + a_1x_1 + b_1x_2 \quad and \quad w^2= c_2 + a_2x_1 + b_2x_2.
\end{align*}

We will only focus on this region now.
\bp \label{GSpray} Suppose $w$ is an affine weak external field as in above. Then the geodesic spray coefficients that provide the optimal control trajectories (geodesics for the associated Randers metric) in the Zermelo navigation problem, are given by
 \begin{align}\label{Sp1}
	G^1(x,\mathbf{y}) &= \left( FP^1_1 \right) y^1+\left( FP^1_2 \right) y^2 + \left( L^1_1 \right) \left(y^1 \right)^2 + (L^1_2) \left(y^2 \right)^2+(Q_0^1) y^1  y^2 \notag \\
	&+ \left( \frac{A^1_0}{F} \right) \left(y^1 \right)^3+ \left( \frac{B^1_0}{F} \right) \left(y^1 \right)^2 y^2 + \left( \frac{D^1_0}{F} \right) y^1 \left(y^2 \right)^2 +F^2R_0^1.
\end{align}
And
\begin{align}\label{Sp2}
	G^2(x,\mathbf{y}) &= \left( FP^2_1 \right) y^1 +( FP^2_2 )y^2+ \left( L^2_1\right) \left(y^1 \right)^2 + L^2_2 \left(y^2 \right)^2+(Q_0^2)y^1y^2
	\notag \\
	&+ \left( \frac{A^2_0}{F} \right) \left(y^1 \right)^3+ \left( \frac{B^2_0}{F} \right) \left(y^1 \right)^2 y^2 + \left( \frac{D^2_0}{F} \right) y^1 \left(y^2 \right)^2+F^2R_0^2,
\end{align}
where $F$ is the Randers Finsler metric given in (\ref{eq:Randers}),
$L$ 's and $Q$'s are polynomials of degree 1 , $P$'s are polynomials of degree 2, $R$'s are polynomials of degree 3 all in terms of the space variables $x_i$ ($i=1,2$) and the rest of the coefficients are constants. Of course all these depend on $c_i, a_i, b_i$, $i=1,2$. See the proof for precise expression for these polynomials.
\ep
\begin{proof}

According to the formulas (\ref{eq:s-and-t-1}) and (\ref{eq:s-and-t-2}), we have

\begin{align*}
S_1 + T_1 =  2a_1(c_1+a_1x_1 + b_1x_2) + 2a_2 (c_2+a_2x_1 + b_2x_2).
\end{align*}
\begin{align*}
S_2 + T_2 =  2b_1(c_1+a_1x_1 + b_1x_2) + 2b_2 (c_2+a_2x_1 + b_2x_2).
\end{align*}
Thus $S_i + T_i$, $i=1,2$ are affine. Similarly
\begin{align*}
L_{00} = 2a_1(y^1)^2 + 2(a_2 + b_1) y^1 y^2 + 2b_2 (y^2)^2,
\end{align*}
which is a quadratic form in terms of $y^1$ and $y^2$. One also computes
\begin{align*}
 S_0 &= S_1y^1 + S_2y^2 \\ &= \left[ 2w^1w^1_{x_1} + w^2(w^1_{x_2} + w^2_{x_1}) \right] y^1 \\ & \ + \left[ w^1(w^1_{x_2} + w^2_{x_1}) + 2w^2w^2_{x_2} \right] y^2 \notag \\ &= \left[ 2(c_1+a_1x_1 + b_1x_2)a_1 + (c_2+a_2x_1+ b_2x_2)(a_2 + b_1) \right] y^1 \\ & \ + \left[ (c_1+a_1x_1 + b_1x_2)(a_2 + b_1) + 2(c_2+a_2x_1 + b_2x_2)b_2 \right] y^2 \notag \\ &= \left[ (2a_1^2 + a^2_2 + a_2 b_1)x_1 + (2a_1b_1 + a_2b_2 + b_1b_2)x_2 +(2c_1a_1+c_2a_2+c_2b_1)\right] y^1 \notag \\ & \:\: + \left[  (a_1a_2 + a_1b_1 + 2a_2b_2 )x_1 + (b_1a_2 + b_1^2 + 2b_2^2)x_2+(c_1a_2+c_1b_1+2c_2b_2) \right] y^2
\end{align*}
In addition, we have
\begin{align*}
L_{ww} &= 2 (w^1)^2w^1_{x_1} + 2w^1w^2(w^1_{x_2} + w^2_{x_1}) + 2(w^2)^2 w^2_{x_2} \notag \\ &= 2(c_1+a_1x_1 + b_1x_2)^2a_1 + 2( c_1+a_1x_1 + b_1x_2)( c_2+a_2x_1 +b_2x_2)(b_1+a_2)\\
&\:\:\:\: + 2 (c_2+a_2x_1 +b_2x_2)^2b_2, \notag
\end{align*}
is an affine translation of a quadratic form in terms of $x_i$, $i=1,2$.

\par To make the formulas simpler, we set (all $A$'s, $M$'s and $N$ below are constants)
\begin{align*}
A_1&=a_1,\: B_1=b_1 ,\: C_1=c_1,\\
A_2&=a_2,\: B_2=b_2,\: C_2=c_2, \\
A_3&=2(a_1^2+a_2^2),\: B_3=2(a_1b_1+a_2b_2), \: C_3=2(a_1c_1+a_2c_2), \\
A_4&=2(a_1b_1+a_2b_2),\: B_4=2(b_1^2+b_2^2),\: C_4= 2(b_1c_1+b_2c_2), \\
A_5&=a_2^2+b_1a_2+2a_1^2,\: B_5=a_2b_2+b_1b_2+2a_1b_1,\: C_5=2c_1a_1+c_2a_2+c_2b_1, \\
A_6&=2a_2b_2+a_1b_1+a_1a_2,\: B_6=2b_2^2+b_1^2+b_1a_2),\: C_6=c_1a_2+c_1b_1+2c_2b_2,\\
E&=2a_1,\: J=2(a_2+b_1),\: K=2b_2,\\
M_{11}&=2(2a_1^2b_1+a_1b_1b_2+a_1a_2b_2+b_1^2a_2+b_1a_2^2+2a_2b_2^2),\\
M_{02}&=2(b_1^2a_1+b_1^2b_2+a_2b_1b_2+b_2^3),\\
M_{20}&=2(a_1^3+a_1b_1a_2+a_1a_2^2+a_2^2b_2),\\
M_{00}&=2(c_1^2a_1+c_1c_2b_1+c_1c_2a_2+b_2c_2^2)\\
M_{10}&=2(2c_1a_1^2+c_1a_2b_1+a_1c_2b_1+c_1a_2^2+a_1a_2c_2+2c_2a_2b_2)\\
M_{01}&=2(2c_1a_1b_1+c_1b_2b_1+b_1^2c_2+c_1a_2b_2+b_1c_2a_2+2c_2b_2^2)
\\
N&=(b_1-a_2).
\end{align*}
So we have
\begin{align*}
w^i=C_i+ A_ix_1 + B_ix_2, \quad i=1,2,
\end{align*}
\begin{align*}
S_{(i-2)} + T_{(i-2)} = A_ix_1 + B_ix_2+C_i, \quad i=3,4,
\end{align*}
\begin{align*}
S_0 = \left( A_5 x_1 + B_5x_2 +C_5 \right) y^1+\left( A_6 x_1 + B_6x_2+C_6  \right) y^2,
\end{align*}
\begin{align*}
L_{00} = E(y^1)^2 + Jy^1 y^2 + K (y^2)^2,
\end{align*}
\begin{align*}
L_{ww} = \sum_{l+k\leq2} M_{lk} x_1^{l}x_2^k,
\end{align*}
\begin{align*}
C_0^1=Ny^2,\:\: C_0^2=-Ny^1.
\end{align*}
\begin{align*}
\zeta^1 =&\frac{1}{4}\left( \frac{y^1}{F}-(C_1+A_1x_1+B_1x_2)\right)(2F((A_5x_1+B_5x_2+C_5)y^1+
(A_6x_1\\
\quad \quad \quad &+B_6x_2+C_6)y^2)
-(E(y^1)^2+Jy^1y^2+K(y^2)^2)-F^2\left(\sum_{l+k\leq 2} M_{lk} x_1^{l}x_2^k \right))\\
&-\frac{1}{4}F^2(A_3x_1+B_3x_2+C_3)-\frac{1}{2}FNy^2.
\end{align*}
We set
\begin{align*}
&P^1_1=-\frac{1}{2}(A_1A_5x_1^2+A_1B_5x_1x_2+B_1A_5x_1x_2+B_1B_5x_2^2+(A_5C_1\\
& \quad \quad +A_1C_5)x_1 +(C_1B_5+B_1C_5)x_2+C_1C_5)-\frac{1}{4}\left(\sum_{l+k\leq 2} M_{lk} x_1^{l}x_2^k \right),\\
&L_1^1=\frac{1}{2}(A_5x_1+B_5x_2+C_5)+\frac{1}{4}(A_1Ex_1+B_1Ex_2+C_1E),
\end{align*}
\begin{align*}
&Q_0^1=\frac{1}{2}(A_6x_1+B_6x_2+C_6)+\frac{1}{4}(A_1Jx_1+B_1Jx_2+C_1J),\\
&P_2^1= -\frac{1}{2}(A_1A_6x_1^2+(A_1B_6+B_1A_6)x_1x_2+B_1B_6x_2^2+(C_1A_6+A_1C_6)x_1\\
&\quad\quad +(C_1B_6+B_1C_6)x_2+C_1C_6+N),
\end{align*}
\begin{align*}
&L_2^1=\frac{1}{4}(A_1Kx_1+B_1Kx_2+C_1K),\\
&R_0^1=\frac{1}{4}(A_1x_1\sum_{l+k\leq 2} M_{lk} x_1^{l}x_2^k+B_1x_2\sum_{l+k\leq 2} M_{lk} x_1^{l}x_2^k+C_1\sum_{l+k\leq 2} M_{lk}x_1^{l}x_2^k
\\
&\quad \quad-(A_3x_1+B_3x_2+C_3)),\\
&A_0^1=-\frac{E}{4},\\
 & B_0^1=-\frac{J}{4},\\
 & D_0^1=-\frac{K}{4},
\end{align*}
where $A_0,B_0,D_0$ are constants, $L^i_j(x_1,x_2)$ are degree one polynomials and $P^i_j(x_1,x_2)$ are degree two polynomials, and $R^k_l$ are polynomials of degree three in terms of $x_1,x_2$.

 A similar statement holds for $\zeta^2$ and upon combining these expressions, we are able to write down the geodesic spray coefficients and get the formula \ref{Sp1} and \ref{Sp2}.
\end{proof}

\par At this point, having computed the geodesic spray coefficients in Proposition \ref{GSpray}, we are in a position to obtain the precise equations for the geodesics.

Suppose $\left( x_1(t), x_2(t) \right)$ is a geodesic in the rectangular domain as set in above.  The equations $x_i$  are given as follows.
\setcounter{thm}{0}
\bt
(Extended form)
 Let $ w $ be a weak dynamic linear vector in $ \mathbb{R}^2$ of the form $ w=(w^1,w^2)=(c_1+a_1x_1+b_1x_2,c_2+a_2x_1+b_2x_2) $. The unit speed geodesic (with respect to F)
 starting from initial position to its destination is given by the formula
 \begin{align}\label{geo1}
 	&\ddot{x_1} + 2 ( \left( FP^1_1 \right) \dot{x_1}+\left( FP^1_2 \right) \dot{x_2} + \left( L^1_1 \right) \left(\dot{x_1} \right)^2 + (L^1_2) \left(\dot{x_2} \right)^2+(C_0^1) \dot{x_1}  \dot{x_2} \\ \notag
 	&+ \left( \frac{A^1_0}{F} \right) \left(\dot{x_1} \right)^3+ \left( \frac{B^1_0}{F} \right) \left(\dot{x_1} \right)^2 \dot{x_2} + \left( \frac{D^1_0}{F} \right) \dot{x_1} \left(\dot{x_2} \right)^2 +F^2R_0^1
 	) = \frac{\dot{F}}{F} \dot{x}_1.
 \end{align}
 And
 \begin{align}\label{geo2}
 	&\ddot{x_2} + 2 (\left( FP^2_1 \right) \dot{x_1} +( FP^2_2 )\dot{x_2}+ \left( L^2_1\right) \left(\dot{x_1} \right)^2 + L^2_2 \left(\dot{x_2} \right)^2+(C_0^2)\dot{x_1}\dot{x_2}
 	\\ \notag
 	&+ \left( \frac{A^2_0}{F} \right) \left(\dot{x_1} \right)^3+ \left( \frac{B^2_0}{F} \right) \left(\dot{x_1} \right)^2 \dot{x_2} + \left( \frac{D^2_0}{F} \right) \dot{x_1} \left(\dot{x_2} \right)^2+F^2R_0^2.
 	)= \frac{\dot{F}}{F} \dot{x_2}.
 \end{align}
\et
\begin{proof}
By using the geodesic spray coefficients in Proposition \ref{GSpray} we have \begin{align*}
	&\ddot{x_1} + 2 ( \left( FP^1_1 \right) y^1+\left( FP^1_2 \right) y^2 + \left( L^1_1 \right) \left(y^1 \right)^2 + (L^1_2) \left(y^2 \right)^2+(Q_0^1) y^1  y^2 \\
	&+ \left( \frac{A^1_0}{F} \right) \left(y^1 \right)^3+ \left( \frac{B^1_0}{F} \right) \left(y^1 \right)^2 y^2 + \left( \frac{D^1_0}{F} \right) y^1 \left(y^2 \right)^2 +F^2R_0^1
	) = \frac{\dot{F}}{F} \dot{x}_1,
\end{align*}
notice that with our notation, $\dot{x_i} = y^i$ so we get the formula (\ref{geo1}) and similarly (\ref{geo2}).
\end{proof}

Here, we get two nonlinear ODEs that are simpler than the geodesic equations for the general external field $w$.

\subsection{A few words on applications}
Here, we just highlight the importance of the equations we have obtained in computing the optimal trajectories. Since this is a theoretical work, we have not included numerical solutions of the equations obtained. The numerical considerations will be carried out in our upcoming works.

Let us note that since the coefficients are simpler, solving these equations by numerical methods would be an easier task than solving the geodesic equations in the general case; Also solving these equations numerically would be simpler than solving the optimal control problem (which has the same complexity as solving the geodesic problem in the general case).

When the domain of these equations are assumed to be very small, one can further simplify the equations at hand by assuming $F$ only depends on $y^i$s. This makes sense since the tangent space at a point in a Finslerian manifold is a Minkowski space in which the norm does not depend on the base point.

\vspace{5mm}
\noindent Faculty of Mathematics and Computer Sciences,
Amirkabir University of Technology (Tehran Polytechnic),
424 Hafez Ave. 15914 Tehran, Iran.

\noindent \href{mailto:bidabad@aut.ac.ir}{bidabad@aut.ac.ir};   \href{mailto:z.fathi@aut.ac.ir}{z.fathi@aut.ac.ir}  \\

\noindent  Institut de Math\'{e}matique de Toulouse, Universit\'{e} Paul Sabatier,
F-31062 Toulouse, France.

\noindent \href{mailto:behroz.bidabad@math.univ-toulouse.fr}{behroz.bidabad@math.univ-toulouse.fr}\\

\end{document}